\documentclass[11pt]{article}

\usepackage[affil-it]{authblk}

\usepackage{amsfonts}
\usepackage{amsmath}
\usepackage{amssymb}
\usepackage{amsthm}
\usepackage{amsbsy}
\usepackage{xcolor}
\usepackage{graphicx} \usepackage{enumerate} \usepackage{multicol}
\usepackage{mathrsfs} \usepackage[all,cmtip]{xy}
\usepackage{todonotes}

\newcounter{dummy} \numberwithin{dummy}{section}
\newtheorem{theorem}[dummy]{Theorem}
\newtheorem{corollary}[dummy]{Corollary}
\newtheorem{lemma}[dummy]{Lemma}

\newtheorem{proposition}[dummy]{Proposition}
\theoremstyle{remark}
\newtheorem{remark}[dummy]{Remark}
\newtheorem{example}[dummy]{Example}

\newcommand{\calF}{\mathcal{F}}
\newcommand{\calH}{\mathcal{H}}

\newcommand{\calL}{\mathcal{L}}
\newcommand{\calV}{\mathcal{V}}
\newcommand{\M}{\mathbb{M}}

\newcommand{\Dhe}{\Delta_{\mathcal{H},\ve}}
\newcommand{\ves}{{\ve^*}}

\newcommand{\ch}{\mathcal{H}}

\newcommand{\scrR}{\mathscr{R}}
\newcommand{\scrT}{\mathscr{T}}

\DeclareMathOperator{\spn}{span}

\DeclareMathOperator{\tr}{tr}

\newcommand{\ve}{\varepsilon}

\DeclareMathOperator{\ptr}{/\! \! /}

\DeclareMathOperator{\Lbra}{[ \! [}
\DeclareMathOperator{\Rbra}{] \! ]}

\numberwithin{equation}{section}

\title{A horizontal Chern-Gauss-Bonnet formula on totally geodesic foliations}

\author{Fabrice Baudoin%
  \thanks{Author supported in part by the NSF Grant DMS 1901315}}
\affil{Department of Mathematics, University of Connecticut, USA\par
 \texttt{fabrice.baudoin@uconn.edu}}

\author{Erlend Grong}
\affil{University of Bergen, Department of Mathematics, P.O.~Box 7803, 5020 Bergen, Norway\par
\texttt{erlend.grong@uib.no}}

\author{Gianmarco Vega-Molino%
 }
\affil{Department of Mathematics, University of Connecticut, USA\par
 \texttt{gianmarco.molino@uconn.edu}}

\begin{document}

\maketitle

\begin{abstract}
Under suitable conditions, we show that the Euler characteristic of a foliated Riemannian manifold can be computed only from curvature invariants which are transverse to the leaves. Our proof uses the hypoelliptic sub-Laplacian on forms recently introduced by two of the authors in \cite{BaGr19}.
 \end{abstract}


\section{Introduction}

The goal of the paper is to prove the following result:

\begin{theorem}\label{theo-intro}
 Let $\M$ be a smooth, connected, oriented and $n+m$ dimensional compact  manifold. We assume that $\M$ is equipped with a  Riemannian foliation $\mathcal{F}$ with bundle-like  metric $g$ and totally geodesic $m$-dimensional leaves.  We also assume that the horizontal distribution $\mathcal{H}=\mathcal{F}^\perp$ is bracket-generating and that there exists $\ve >0$ such that  
\begin{align}\label{J-intro}
 (\nabla_v J)_w = - \frac{1}{2\ve} [J_v, J_w]
 \end{align}
 for any $v,w \in T_x\M$, $x \in \M$, where $\nabla$ is the Bott connection of the foliation and $J$ is the tensor defined in \eqref{Jmap}. Denoting $\chi(\M)$ the Euler characteristic of $\M$:
 \begin{itemize}
 \item If $n$ or $m$ is odd, then $\chi (\M)=0$;
 \item If $n$ and $m$ are both even then
 \[
\chi (\M)=  \int_\M     \hat{\omega}_\mathcal{H}^\ve \wedge \left[  \det \left( \frac{\scrT}{\sinh (\scrT )}\right)^{1/2}\right]_m .
\]
 \end{itemize}
\end{theorem}
%

Notations are further explained in section 4 but we point out that a remarkable feature of that result is that the density $ \hat{\omega}_\mathcal{H}^\ve \wedge \left[  \det \left( \frac{\scrT}{\sinh (\scrT)}\right)^{1/2}\right]_m$ essentially only depends on \textit{horizontal curvature quantities}. Therefore the theorem illustrates further the fact already observed in \cite{BaGr19} that topological properties of $\M$ might be obtained from horizontal curvature invariants only provided that the bracket-generating condition of the horizontal distribution is satisfied; thus, in essence, the theorem is a sub-Riemannian result. We also note that the condition \eqref{J-intro} is satisfied in a large class of examples including the H-type foliations introduced in \cite{BGMR18}, see Example \ref{HtypeEx}.

The proof of Theorem \ref{theo-intro}  is based on the study of the heat semigroup generated by the hypoelliptic sub-Laplacian on forms recently introduced in \cite{BaGr19}. The heat equation approach to Chern-Gauss-Bonnet type formulas (or index formulas) that we are using is of course not new:  It was suggested by Atiyah-Bott \cite{MR212836} and McKean-Singer \cite{MR217739} and first carried out by Patodi \cite{MR290318} and Gilkey \cite{MR324731} and is by now classical, see the book \cite{MR2273508}. However a difficulty in our setting is that the sub-Laplacian on forms we consider is only hypoelliptic but not elliptic. To carry out the required small-time asymptotics analysis to obtain the horizontal Chern-Gauss-Bonnet formula, we will make use of the probabilistic Brownian Chen series parametrix method first introduced in \cite{index} and which is easy to adapt to hypoelliptic situations, see \cite{bismutCR}.

The paper is organized as follows. In section 2 we introduce the horizontal Laplacian on forms $\Delta_{\calH,\ve}$  and prove that it is a self-adjoint operator if and only if the condition \eqref{J-intro} is satisfied. In section 3, we prove a McKean-Singer type formula for $\Delta_{\calH,\ve}$, namely that for every $t > 0$,
\[
\mathbf{Str} ( e^{t \Dhe }) =\chi (\M).
\]
Finally, in section 4 we study the small-time asymptotics of $\mathbf{Str} ( e^{t \Dhe }) $ and conclude the proof of Theorem \ref{theo-intro}.

%
%
%

\section{Preliminaries}

In this section, we first recall the framework and notations of \cite{BaGr19} and the references therein to which we refer for further details. We then prove a necessary and sufficient condition for the form horizontal Laplacian of a totally geodesic foliation to be a symmetric operator. 

\subsection{Totally geodesic foliations}

Let~$(\M,g)$ be a smooth, oriented, connected, compact Riemannian manifold with dimension $n+m$. We assume that~$\M$ is equipped with a foliation $\mathcal{F}$ with  $m$-dimensional leaves. The distribution $\mathcal{V}$ formed by vectors tangent to the leaves is referred  to as the set of \emph{vertical directions} (or \emph{vertical subbundle}). Define \emph{the horizontal subbundle} $\calH = \calV^\perp$ as its orthogonal complement. We will always assume in this paper that the horizontal distribution $\mathcal{H}$ is everywhere  bracket-generating. The foliation is called \emph{Riemannian} and \emph{totally geodesic} if for any $X \in \Gamma(\calH)$, $Z \in \Gamma(\calV)$, the respective conditions are satisfied,
$$(\calL_Z g)(X,X) =0, \qquad (\calL_X g)(Z,Z) =0.$$

Equivalently, we can describe these conditions using \emph{the Bott connection}. Write $\pi_{\calH}$ and $\pi_{\calV}$ for the respective orthogonal projections to $\calH$ and $\calV$. Let $\nabla^g$ be the Levi-Civita connection of $g$. Introduce a new connection $\nabla$ on $T\M$ according to the rules,
$$\nabla_X Y = \left\{ \begin{array}{ll}
\pi_\calH( \nabla^g_X Y) & \text{for any $X, Y \in \Gamma(\calH)$,} \\
\pi_\calH ([X,Y] )& \text{for any $X \in \Gamma(\calV)$, $Y \in \Gamma(\calH)$,} \\
\pi_\calV ([X,Y] )& \text{for any $X \in \Gamma(\calH)$, $Y \in \Gamma(\calV)$,} \\
\pi_\calV (\nabla^g_X Y) & \text{for any $X, Y \in \Gamma(\calV)$.} 
\end{array}\right.$$
We observe that $\nabla$ preserves $\calH$ and $\calV$ under parallel transport. The foliation $\calF$ is then both Riemannian and totally geodesic if and only if $\nabla g = 0$. For the rest of the paper, we will assume that $\nabla$ is indeed compatible with the metric $g$. The torsion $T$ of $\nabla$ is given by
$$T(X, Y) = - \pi_{\calV} [\pi_{\calH} X, \pi_{\calH} Y ].$$
Define a corresponding endomorphism valued one-form $Z \mapsto J_Z$ by
\begin{equation} \label{Jmap} \langle J_Z X, Y \rangle_g = \langle Z, T(X,Y) \rangle_g, \qquad X,Y, Z \in \Gamma(T\M).\end{equation}
Let $g_\calH$ and $g_\calV$ be the respective restrictions of $g$ to $\calH$ and $\calV$. We then define \emph{the canonical variation} $g$ by $g_\ve = g_\calH \oplus \frac{1}{\ve} g_\calV$, $\ve >0$,  and make the following observations:
\begin{enumerate}[\rm (i)]
\item If $(\M , \calF, g)$ is a Riemannian, totally geodesic foliation, then so is $(\M, \calF, g_\ve)$.
\item If we define the Bott connection according to $g_\ve$, then the resulting connection coincide with the one defined relative to $g$.
\item For any fixed $\ve >0$, define a connection
\begin{equation} \label{hatnablave} \hat \nabla^\ve_XY = \nabla_X Y + \frac{1}{\ve} J_X Y.\end{equation}
This connection preserves $\calH$ and $\calV$ under parallel transport and is compatible with $g_{\ve_2}$ for any $\ve_2 >0$. Furthermore, its torsion
$$\hat T^\ve(X,Y) = T(X,Y) + \frac{1}{\ve} J_X Y - \frac{1}{\ve} J_Y X,$$
is skew-symmetric with respect to $g_\ve$. Hence, if we consider its adjoint connection
\begin{equation} \label{nablave} \nabla_X^\ve Y = \hat \nabla_X^\ve Y - \hat T^\ve(X,Y) = \nabla_X Y - T(X,Y) + \frac{1}{\ve} J_Y X,\end{equation}
it will also be compatible with $g_\ve$. However, $\calH$ and $\calV$ are not parallel with respect to $\nabla^\ve$.
\end{enumerate}

\subsection{Horizontal Laplacian on forms}

For the totally geodesic Riemannian foliation $(\M, \calF, g)$, define its horizontal Laplacian on functions $ f \in C^\infty(\M)$ by
\begin{equation} \label{HorLaplacianFunction}\Delta_{\calH} f = \tr_{\calH} \nabla_{\times}df(\times).\end{equation}
We note that since $\calH$ is assumed to be bracket-generating, from H\"ormander's theorem,  $\Delta_{\calH}$ is a subelliptic operator. We  also note that since $g_\calH$ and the Bott connection will coincide for any $g_\ve$, $\ve >0$, the horizontal Laplacian will coincide for any metric in the canonical variation family.

Consider now the totally geodesic Riemannian foliation $(\M , \calF, g_\ve)$ for some fixed $\ve >0$. We want to extend the horizontal Laplacian on functions \eqref{HorLaplacianFunction} to a differential operator on forms $\Delta_{\calH,\ve}$ satisfying the following requirements:
\begin{enumerate}[\rm (I)]
\item $\Delta_{\calH,\ve} f = \Delta_{\calH}f$ for any smooth function $f$;
\item The operator $\Delta_{\calH, \ve}$ is of Weitzenb\"ock type, i.e. $\Delta_{\calH, \ve} = L_{\calH, \ve} + \scrR_\ve$ where $\scrR_\ve$ is a zero-order differential operator and
\begin{equation} \label{ConnectionLaplacian} L_{\calH,\ve} =\tr_{\calH} \tilde \nabla_{\times, \times}^2,\end{equation}
is the connection horizontal Laplacian of some connection $\tilde \nabla$ compatible with~$g_\ve$;
\item If $d$ is the exterior differential, then
$$[\Delta_{\calH,\ve}, d] =0.$$
\end{enumerate}
Given these requirements, there is an essentially unique extension of $\Delta_{\calH}$ to forms, see \cite{BaGr19, GrTh19} for details. We call $\Delta_{\calH,\ve}$ the $\ve$-horizontal Laplacian on forms. This operator can described as follows.

\begin{proposition}[Horizontal Laplacian on forms, see \cite{BaGr19}]
Consider the  $\ve$-horizontal divergence operator defined by
 $$\delta_{\calH,\ve} \eta = - \tr_{\calH} (\nabla_\times^\ve \eta)(\times, \cdot  ).$$
 The operator
$$\Delta_{\calH,\ve} = - \delta_{\calH,\ve} d - d\delta_{\calH,\ve},$$
is called the $\ve$-horizontal Laplacian on forms and it satisfies the  requirements $\mathrm{(I)}, \mathrm{(II)}, \mathrm{(III)}$. In particular, this operator has Weitzenb\"ock decomposition $\Delta_{\calH, \ve} = L_{\calH, \ve} + \scrR_\ve$ where $L_{\calH,\ve}$ is defined as in \eqref{ConnectionLaplacian} relative to $\nabla^\ve$.
\end{proposition}
We can describe the zero order operator $\scrR_\ve$ can be made explicit, see \cite{BaGr19}. For later use, we will prefer to write the operators using Fermion calculus, see Appendix~\ref{sec:Fermion}. Let $X_1, \dots, X_n$ and $Z_1, \dots, Z_m$ be local orthonormal bases of respectively $\calH$ and $\calV$. Define $a_i = \iota_{X_i}$ and $b_r = \iota_{Z_r}$ for the corresponding annihilation operators, with the dual operators $a^*_i = X^*_i \wedge$ and $b_r^* = Z_r^* \wedge$ a acting by wedge products. The dual are here relative to the $L^2$ inner product with respect to the fixed metric $g$. Relative to the curvature tensor $\hat R^\ve$ of $\hat \nabla^\ve$, write
\begin{equation} \label{indices} \hat R^{\ve,l}_{ijk} = \langle \hat R^\ve(X_i, X_j) X_k, X_l \rangle_g,\end{equation}
and use similar notation for other tensors with indices $i,j,k,l$ denoting evaluations with respect to the basis of $\calH$, indices $r,s$ with respect to the basis of $\calV$. We emphasize that these indices are always defined relative to the fixed metric $g$. Then $\mathscr{R}_\ve$ is given by
\begin{align} \label{scrR1}
\scrR_\ve &= \sum_{i,j,k=1}^n \hat R_{ijk}^{\ve,i} a_k^* a_j   + \sum_{i,k=1}^n \sum_{r=1}^m  \hat R_{irk}^{\ve,i}  a_k^* b_r + \frac{1}{2} \sum_{i,j,k,l=1}^n   \hat R_{ijk}^{\ve,l} a_k^* a_l^* a_j a_i  \\ \nonumber
& \qquad +  \sum_{i,j,k=1}^n \sum_{r=1}^m \hat R_{irk}^{\ve,l} a_k^* a_l ^*b_r a_i  + \frac{1}{2}  \sum_{i,j=1}^n \sum_{r,s=1}^m \hat R_{rsi}^{\ve,j} a_i^* a_j^* b_r b_s.
\end{align}
We want to give a formula for this operator that shows the dependence of $\ve$ explicitly. Let $T$ and $R$ be the curvature of the Bott connection $\nabla$,  and use indices after semi-colons to denote covariant derivatives with respect to this connection. Using Lemma~\ref{lemma:hatRtoR}, Appendix, we can write
\begin{align} \label{scrR}
\scrR_\ve &= \sum_{i,j,k=1}^n \left(R_{kji}^k + \frac{1}{\ve} \sum_{r=1}^m T_{ik}^r T^r_{jk} \right)  a_i^* a_j  - \sum_{i,j=1}^n \sum_{r=1}^m T_{ij;i}^r a_j^* b_r \\ \nonumber
& \qquad + \frac{1}{2} \sum_{i,j,k,l=1}^n  \left(R_{kli}^j +\frac{1}{\ve} \sum_{r=1}^m T_{kl}^r T_{ij}^r \right)  a_i^* a_j^* a_l a_k  +  \sum_{i,j,k=1}^n \sum_{r=1}^m \frac{1}{\ve} T_{ij;k}^r  a_i^* a_j^* b_r a_k \\ \nonumber
& \qquad + \frac{1}{2}  \sum_{i,j=1}^n \sum_{r,s=1}^m \left( \frac{2}{\ve} T_{ij;r}^s + \frac{1}{\ve^2} \sum_{k=1}^n (T_{kj}^r T_{ik}^s - T_{kj}^s T_{ik}^r) \right) a_i^* a_j^* b_s b_r.
\end{align}

\subsection{Symmetry of the horizontal Laplacian} \label{sec:Symmetry}

Consider the exterior algebra
$$\Omega = \Omega(\M) = \bigoplus_{k=0}^{\dim \M} \Omega^k,$$
with the $L^2$-inner product from $g_\ve$. When restricted to elements in $\Omega^0 \oplus \Omega^1$, the operator $\Delta_{\ch,\ve}$ is symmetric if and only if $\ch$ satisfies the Yang-Mills condition, i.e. if
$$\sum_{i=1}^n T_{ij;i}^r =0,\qquad \text{for any $j=1, \dots, n$, $r= 1, \dots, m$.}$$
see \cite{BKW16}.  Considering all forms, we have the following result.
\begin{proposition} \label{prop:Symmetric}
The operator $\Delta_{\ch,\ve}$ is symmetric with respect to the $L^2$-inner product of $g_\ve$ if and only if
\begin{equation} \label{SymmetryCond} (\nabla_v J)_w = - \frac{1}{2\ve} [J_v, J_w],\end{equation}
for any $v,w \in T_xM$, $x \in M$. In particular, $\nabla_v J = 0$ for any $v \in \calH$.
\end{proposition}
We note that under the above condition, the expression of $\scrR_\ve$ reduces to
\begin{align} \label{scrR3}
\scrR_\ve &= \sum_{i,j,k=1}^n \left(R_{kji}^k + \frac{1}{\ve} \sum_{r=1}^m T_{ik}^r T^r_{jk} \right)  a_i^* a_j + \frac{1}{2} \sum_{i,j,k,l=1}^n  \left(R_{kli}^j +\frac{1}{\ve} \sum_{r=1}^m T_{kl}^r T_{ij}^r \right)  a_i^* a_j^* a_l a_k.
\end{align}

\begin{proof}
$L_{\calH,\ve}$ is symmetric by \cite[Lemma~A.1]{GrTh19}, so we only need to determine when $\scrR_\ve$ is symmetric. We choose a local bases $X_1, \dots, X_n$ and $Z_1, \dots, Z_m$ of respectively $\calH$ and $\calV$. We consider the representation of $\scrR_\ve$ as in \eqref{scrR}. Then for $\scrR_\ve$ to be symmetric, we must have
\begin{align*}
0 &= \langle \scrR_\ve X_k^* \wedge Z_r^*, X_i^* \wedge X^j \rangle_{\ve} - \langle \scrR_\ve X_i^* \wedge X_j, X_k^* \wedge Z_r^* \rangle_{\ve} = \frac{1}{\ve} T_{ij;k}^r, \\
0& = \langle \scrR_\ve Z_r^* \wedge Z_s^*, X_i^* \wedge X^j \rangle_{\ve} - \langle \scrR_\ve X_i^* \wedge X_j, Z_r^* \wedge Z_s^* \rangle_{\ve} \\
& = \frac{2}{\ve} T_{ij;r}^s + \frac{1}{\ve^2} \sum_{k=1}^n (T_{kj}^r T_{ik}^s - T_{kj}^s T_{ik}^r).
\end{align*}
These equations are clearly equivalent to \eqref{SymmetryCond}. If these hold, then $\scrR_\ve$ reduces to the expression \eqref{scrR3}, which is symmetric by Lemma~\ref{hatRtoR}~(i).
\end{proof}

\begin{remark}
If we assume that $m=1$ (i.e. the leaves are one-dimensional), then it is immediate from the previous result that the following are equivalent:
\begin{enumerate}[\rm (i)]
\item $\Delta_{\calH, \ve}$ is symmetric for some $\ve >0$.
\item $\Delta_{\calH, \ve}$ is symmetric for all $\ve >0$.
\item $\nabla J=0$.
\end{enumerate}
\end{remark}

\begin{example}[H-type foliations] \label{HtypeEx}
Following definitions given in \cite{BGMR18}, we say that a foliated Riemannian manifold $(\M, \calF, g)$ is of \emph{H-type} if for every $Z \in \Gamma(\calV)$, we have $J_Z^2 = - \| Z \|_{\calV}^2 \pi_{\calH}$. Expand the definition of $J$ from taking values from $\calV$ to its Clifford algebra $\mathbf{Cl}(\calV)$ by the rule $J_1 = \pi_{\calH}$ and iteratively $J_{u \cdot v} = J_u J_v$, $u, v \in \mathbf{Cl}(\calV)$. We then further say that the foliation is of horizontally parallel Clifford type if $\nabla_X J =0$ for any horizontal vector fields $X \in \Gamma(\calH)$ and while for $u,v \in \calV$.
$$(\nabla_{u}J)_v \in J_{\mathbf{Cl}(\calV)}.$$
It then turns out that for some $\kappa \in \mathbb{R}$,
$$(\nabla_{u}J)_v = - \kappa J_{u \cdot v + \langle u, v \rangle} = - \frac{\kappa}{2} [J_u, J_v].$$
The number $\kappa$ determines the Ricci curvature of $\nabla$, see \cite[Theorem 3.16]{BGMR18}. We see that if we have an H-type Riemannian foliation $(\M, \calF, g)$ of horizontally parallel Clifford type, then $\Delta_{\calH,\ve}$ is symmetric with respect to $g_\ve$ for $\ve = \frac{1}{\kappa}$.
\end{example}

Finally, to conclude the section we point out the following result. For the definition of the Carnot-Carath\'eodory metric $d_{cc}$ of the sub-Riemannian manifold $(\M, \calH, g_{\calH})$ and the tangent cone of a metric space, see e.g. \cite{Gro96}.

\begin{corollary} \label{cor:SurjV}
Assume that $\Delta_{\calH,\ve}$ is symmetric on forms for some fixed $\ve >0$. Then the following holds:
\begin{enumerate}[\rm (a)]
\item The horizontal bundle $\calH$ has step~$2$, that is $\calH + [\calH, \calH] = T\M$. In particular, the torsion $T$ of the Bott connection $\nabla$ will be surjective on $\calV$.
\item The tangent cones of the metric space $(\M, d_{cc})$ at any pair of points $x, y \in \M$ are isometric.
\end{enumerate}
\end{corollary}

\begin{proof}
\begin{enumerate}[\rm (a)]
\item Recall that if $\Delta_{\calH,\ve}$ is symmetric on forms for some $\ve >0$, then in particular $\nabla_v J =0$ for any $v \in \calH$. We can rewrite it as $\nabla_v T = 0$ for any $v \in \calH$ since $\nabla$ is compatible with $g$. Define $\calH^2 = \calH + [\calH , \calH]$ and let $X_1, X_2, X_3 \in \Gamma(\calH)$ be arbitrary. We first see that
$$T(X_2, X_3) = \nabla_{X_2} X_3 - \nabla_{X_3} X_2 - [X_2,X_3] =0 \mod \calH^2,$$
since $\nabla$ preserves $\calH$. Furthermore, by the definition of the Bott connection
\begin{align*}
& [X_1, [X_2, X_3]] = - [X_1, T(X_2, X_3)] \mod \calH^2 = - \nabla_{X_1} T(X_2, X_3) \mod \calH^2 \\
& = -  T(\nabla_{X_1} X_2, X_3)  -  T(X_2, \nabla_{X_1} X_3) \mod \calH^2 = 0 \mod \calH^2.
\end{align*}
It follows that $\calH$ only generates $\calH^2$. As we assumed that $\calH$ is bracket generating, we have $\calH^2 = TM$.
\item Since both $\calH$ and $\calH^2 = \calH + [\calH, \calH] = TM$ have constant rank, it follows by \cite{Mit85, Bel96} that the tangent cone at a point $x$ is a Carnot group $G_x$. Its Lie algebra $\mathfrak{g}_x$ is given by
$$\mathfrak{g}_{x} = \mathfrak{g}_{x,1} \oplus \mathfrak{g}_{x,2} = \calH_x \oplus T_xM/\calH_x,$$
where $TM/\calH_x$ is the center, and for $X_x, Y_x \in \calH_x = \mathfrak{g}_{x,1}$ the Lie bracket is defined as
$$\Lbra X_x, Y_x \Rbra = [X,Y] |_x \mod \calH_x.$$
where $X, Y$ are any pair of vector fields extending this vectors. The Carnot group $G_x$ is then the corresponding simply connected Lie group of $\mathfrak{g}_x$ with the sub-Riemannian structure given by left translation of $\mathfrak{g}_x = \calH_x$ and its inner product.

If identify $\mathfrak{g}_x = \calH_x \oplus T_xM /\calH_x$ with $T_xM = \calH_x \oplus \calV_x$ through the map $v \mod \calH_x \mapsto \pi_{\calV_x}(v)$, $v \in T_xM$. Then the Lie bracket becomes,
$$\Lbra v, w \Rbra = - T(v,w), \qquad v, w\in T_xM.$$
Let now $y$ be any other point and let $\gamma:[0,1] \to \M$ be any horizontal curve from $x$ to $y$, which exists form out assumption that $\calH$ satisfies the bracket-generating condition. Then $\nabla_{\dot \gamma(t)} T =0$ for any $t \in [0,1]$, so if we write
$$\ptr_{\gamma,t} = \ptr_t : T_{x} \M \to T_{\gamma(t)} \M,$$ for the parallel transport map along $\gamma$, then this satisfies
$$\ptr_t T(u, v) = T(\ptr_t u,\ptr_t  v), \qquad v,w \in T_x\M.$$
As a consequence, $\ptr_1:\mathfrak{g}_x= T_x \M \to \mathfrak{g}_y =  T_y \M$ is a Lie algebra isomorphism, which can be integrated to a Lie group isomorphism from $G_x$ to $G_y$. Since the parallel transport $\ptr_1$ also maps $\calH_x$ onto $\calH_y$ isometrically, the induced map on Carnot groups is in fact a sub-Riemannian isometry.
\end{enumerate}
\end{proof}

%
%

\section{Horizontal McKean-Singer theorem   } \label{sec:Deformations}

We work on a totally geodesic foliation $(\M,\mathcal{F},g)$ and assume that there is some $0 < \ve < +\infty$ such that horizontal Laplacian $\Dhe$, is symmetric. From Proposition \ref{prop:Symmetric}, this assumption is equivalent to the fact that
$$(\nabla_v J)_w = -\frac{1}{2\ve} [J_v, J_w].$$

Since $\Dhe$ commutes with $d$ on smooth forms and is symmetric it also commutes on smooth forms with the coderivative $\delta_\ve$, and thus it also commutes with the Hodge-de Rham operator $\Delta_\ve:= -d\delta_\ve -\delta_\ve d$ on smooth forms. From Hodge theorem, the operator $\Delta_\ve$ is elliptic with a compact resolvent and the space of $L^2$-forms can be decomposed as $\oplus_{k=0}^{+\infty} E_{\lambda_k}$ where the $E_{\lambda_k}$'s are the eigenspaces of $\Delta_\ve$. Those eigenspaces only contain smooth forms, therefore $\Dhe (E_{\lambda_k}) \subset E_{\lambda_k}$. This implies that $\Dhe$ is essentially self-adjoint and generates the semigroup:
\begin{align}\label{semi 1}
e^{t \Dhe }=\oplus_{k=0}^{+\infty} e^{t \Dhe \mid E_{\lambda_k}}
 \end{align}
 By hypoellipticity (see \cite[Lemma 4.9]{BaGr19}), this semigroup has a smooth kernel $p_{\mathcal{H}, \ve} (t, x,y)$ and is a bounded  trace class operator in $ L_\mu^2 (\wedge^\cdot \M, g_\ve)$. Let us denote by $E^+_0(\Dhe)$ (resp. $E^-_0(\Dhe)$) the space of harmonic even forms for $\Dhe$ (resp. the space of harmonic odd forms for $\Dhe$).

 The goal of the section is to prove the following theorem, which is an analogue for our horizontal Laplacian of the classical McKean-Singer formula found in \cite{MR217739} :
 
 \begin{theorem}[Horizontal McKean-Singer formula]\label{PG result}
For every $t >0$,
 \begin{align*}
 \mathbf{Str} ( e^{t \Dhe }):& = \int_\M \mathbf{Tr} (p^{+}_{\mathcal{H},\ve} (t,x,x) ) d\mu(x) - \int_\M \mathbf{Tr} (p^{-}_{\mathcal{H},\ve} (t,x,x) ) d\mu(x) \\
  &=\dim E^+_0(\Dhe) -\dim E^-_0(\Dhe) \\
  &=\chi (\mathbb M)
 \end{align*}
 where $\chi (\mathbb M)$ is the Euler characteristic of $\M$.
 \end{theorem}

We turn to the proof of Theorem \ref{PG result}. We denote by 
\[
\mathbf D_\ve=d+\delta_\ve
\]
the Dirac operator of the metric $g_\ve$. Observe that $\mathbf D_\ve$ commutes with $\Dhe$ since both $d$ and $\delta_\ve$ commute with it.  The main idea to prove Theorem \ref{PG result} is to introduce a deformation of $\Dhe $ as follows: 
\[
\square_{\ve, \theta}= (1-\theta)\Dhe -\theta \mathbf  D_\ve^2 , \quad \theta \in [0,1].
\]

A first lemma is the following:

\begin{lemma}
Let $\lambda$ be a non-zero eigenvalue of $\square_{\ve, \theta}$. Then $\mathbf{D}_\ve : E^+_\lambda(\square_{\ve, \theta}) \to  E^-_\lambda (\square_{\ve, \theta})$ is an isomorphism. Therefore, $\dim E^+_\lambda(\square_{\ve, \theta}) =\dim E^-_\lambda (\square_{\ve, \theta})$.
\end{lemma}

\begin{proof}
Let $\lambda$ be a non-zero eigenvalue of $\square_{\ve, \theta}$. The corresponding eigenspace  $E_\lambda(\square_{\ve, \theta})$ is finite-dimensional since $e^{t \square_{\ve, \theta} }$ is a compact operator for $t>0$. Moreover, since $\mathbf D_\ve$ commutes with $\square_{\ve, \theta}$, $\mathbf{D}_\ve : E^+_\lambda(\square_{\ve, \theta}) \to E^-_\lambda (\square_{\ve, \theta})$ is well defined. Let now $\alpha \in E^+_\lambda(\square_{\ve, \theta}) $ such that $\mathbf{D}_{\ve} \alpha=0$. One has then
\[
d\alpha =-\delta_\ve \alpha.
\]
This implies that 
\[
\| d\alpha\|^2_{ L^2 (\wedge^\cdot \M, g_\ve)}= -\langle d\alpha, \delta_\ve \alpha \rangle_{ L^2 (\wedge^\cdot \M, g_\ve)}=0,
\]
so $d\alpha=0$.
Similarly, one has $\|\delta_\ve \alpha \|^2_{ L^2 (\wedge^\cdot \M, g_\ve)}=0,$ so $\delta_\ve \alpha =0$. Therefore, 
\[
\alpha=\frac{1-\theta}{\lambda} \Dhe \alpha=-\frac{1-\theta}{\lambda} (d \delta_{\mathcal{H},\ve} + \delta_{\mathcal{H},\ve} d) \alpha=-\frac{1-\theta}{\lambda} d \delta_{\mathcal{H},\ve}\alpha.
\]
One deduces
\[
\| \alpha\|^2_{ L^2 (\wedge^\cdot \M, g_\ve)}=-\frac{1-\theta}{\lambda} \langle \alpha , d \delta_{\mathcal{H},\ve}\alpha \rangle_{ L^2 (\wedge^\cdot \M, g_\ve)}=-\frac{1-\theta}{\lambda} \langle \delta_\ve \alpha ,  \delta_{\mathcal{H},\ve}\alpha \rangle_{ L^2 (\wedge^\cdot \M, g_\ve)}=0.
\]
As a consequence, $\mathbf{D}_\ve : E^+_\lambda(\square_{\ve, \theta}) \to  E^-_\lambda (\square_{\ve, \theta})$ is injective. Let us now prove that it is surjective. Let $\alpha \in E^-_\lambda (\square_{\ve, \theta})$ which is orthogonal to the space $\mathbf{D}_\ve E^+_\lambda(\square_{\ve, \theta})$. For every $\omega \in E^+_\lambda(\square_{\ve, \theta})$, one has
\[
0=\langle \alpha , \mathbf{D}_\ve \omega \rangle_{ L^2 (\wedge^\cdot \M, g_\ve)}=\langle \mathbf{D}_\ve \alpha ,  \omega \rangle_{ L^2 (\wedge^\cdot \M, g_\ve)}.
\]
Thus, $\mathbf{D}_\ve \alpha=0$ and from the first part of the proof, we deduce that $\alpha=0$. We conclude that $\mathbf{D}_\ve : E^+_\lambda(\square_{\ve, \theta}) \to E^-_\lambda (\square_{\ve, \theta})$ is indeed an isomorphism.
\end{proof}

A second lemma is the following:
\begin{lemma}
For every $t>0$, the map $\theta \to \mathbf{Str} ( e^{t \square_{\ve, \theta} })$ is continuous on $[0,1]$.
\end{lemma}

\begin{proof}
Let $q_{\ve,\theta}(t,x,y)$ be the heat kernel of $\square_{\ve, \theta}=(1-\theta) \Dhe  - \theta \mathbf{D}^2_\ve$, $p_{\mathcal{H},\ve} (t,x,y)$ be the heat kernel of $\Dhe$ and $p_\ve (t,x,y)$ be the heat kernel of $-\mathbf{D}^2_\ve$. Since $-\mathbf{D}^2_\ve$ and $\Dhe$ commute, we have 

\[
e^{t\square_{\ve, \theta} }= e^{t(1-\theta) \Dhe } e^{-t\theta \mathbf{D}^2_\ve}.
\]
Therefore:
\[
q_{\ve,\theta}(t,x,y) =\int_\M  p_{\mathcal{H},\ve} (t(1-\theta),x,z) p_\ve (t \theta ,z,y) dz
\]
and the result easily follows since
\[
\mathbf{Str} ( e^{t \square_{\ve, \theta} })=\int_\M q_{\ve,\theta}(t,x,x) dx.
\]
\end{proof}

We are now ready for the proof of Theorem  \ref{PG result}.

\begin{proof}
From the first lemma:
\begin{align*}
 &  \mathbf{Str} ( e^{t \square_{\ve, \theta} }) \\
 = &\dim E^+_0(\square_{\ve, \theta}) -\dim E^-_0(\square_{\ve, \theta})+ \sum_{ \lambda \neq 0}( \dim E^+_\lambda(\square_{\ve, \theta}) -\dim E^-_\lambda(\square_{\ve, \theta})) e^{\lambda t}  \\
   =& \dim E^+_0(\square_{\ve, \theta}) -\dim E^-_0(\square_{\ve, \theta}).
\end{align*}

Therefore $\mathbf{Str} ( e^{t \square_{\ve, \theta} }) \in \mathbb{Z}$.  From the second lemma $\theta \to \mathbf{Str} ( e^{t \square_{\ve, \theta} })$ is continuous, thus constant. We deduce
\[
\mathbf{Str} ( e^{t \square_{\ve, 0} }) =\mathbf{Str} ( e^{t \square_{\ve, 1} }).
\]
Since $\square_{\ves, 1} = -\mathbf D_\ve^2$ is the Hodge-de Rham Laplacian of the Riemannian manifold $(\M, g_\ve)$, from the usual Riemannian Hodge theory (see \cite{MR217739}), we have
\[
\mathbf{Str} ( e^{t \square_{\ve, 1} })=\chi (\M),
\]
which concludes the proof.
\end{proof}

%
%
%
%

\section{Horizontal Chern-Gauss-Bonnet formula}

As before, we consider the horizontal Laplacian 
\[
\Dhe=-d \delta_{\ch,\ve}-\delta_{\ch,\ve} d,
\]
and assume that it is symmetric for a fixed $\ve$. As seen earlier, $\Dhe$ satisfies the Weitzenb\"ock identity 
\begin{equation}\label{WIdentity}
\Dhe= L_{\calH,\ve} - \scrR_\ve=-(\nabla_\mathcal{H}^\ve)^* \nabla_\mathcal{H}^\ve -\scrR_\ve.
\end{equation}
where the later equality follows from \cite[Lemma~2.1]{GrTh19}. The goal of the section is to compute the pointwise limit
\[
\lim_{t \to 0}   \mathbf{Str}\text{ } (p_{\mathcal{H},\ve} (t,x,x) )
\]
and deduce from it our horizontal Chern-Gauss-Bonnet formula. The computation of that limit will be based on the probabilist method of Brownian Chen series (see \cite{index, ESAIM}) which has the advantage of being easily adapted to subelliptic operators like $\Dhe$, see \cite{bismutCR}. For convenience and to introduce notation, we include in Appendix A.2 the main elements of that theory. 


A first step to implement the method in \cite{bismutCR} is to study the small-time heat kernel asymptotics of a  diffusion tangent to the scalar horizontal Laplacian $\Delta_{\mathcal{H}}$ . Since we assume that $\Dhe$ is symmetric, from Corollary \ref{cor:SurjV} one has $T\M=\mathcal{H} + [\mathcal{H},\mathcal{H}]$, and thus the tangent diffusion will take its values in a two-step Carnot group (the so-called tangent cone, see Corollary \ref{cor:SurjV}(b)) for which an explicit formula for the heat kernel is known (see \cite{cygan,garofalo}).  In a local horizontal frame $\{ X_1,\cdots,X_n \}$ around $x_0$ write
\[V_t(x_0) = \sum_{i=1}^n \sqrt{2} X_i (x_0) B^i_t + \sum_{1 \le i <j \le n} \pi_{\mathcal{V}}( [X_i,X_j] (x_0)) \int_0^t B^i_s dB^j_s -B^j_s dB^i_s,\]
where $(B_t)_{t \ge 0}$ is a Brownian motion in $\mathbb{R}^n$. We note that $V_t(x_0)$ can be written in a basis free way as
\[
 \sqrt{2} B_t(x_0) - \int_0^t T(B_s(x_0), dB_s(x_0))
 \]
 where $B_t(x_0)= \sum_{i=1}^n X_i (x_0) B^i_t $ is a standard Brownian motion in $\calH_{x_0}$.
%

\begin{lemma}\label{densityLemma}
Let $x_0 \in \M$.  For $t>0$, let $d_t (x_0)$ be the density at 0 of the  $T_{x_0} \M$ valued random variable $V_t(x_0)$. Then, when $t \to 0$,
\[
d_t (x_0) \sim \frac{2^m}{(4\pi t)^{\frac{n}{2}+m} } \int_{\mathcal{V}_{x_0}} \det \left( \frac{\sqrt{J_z^* J_z}}{\sinh \sqrt{J_z^* J_z}}\right)^{1/2} dz.
\]
\end{lemma}

\begin{proof}
The process $(V_t(x_0))_{t \ge 0}$ is the horizontal Brownian motion in the tangent cone $G_{x_0}$ which is a 2-step Carnot group when it is identified with $T_{x_0}\M$ using the group exponential map. The heat kernel of the horizontal Laplacian is known explicitly in 2-step Carnot groups (see \cite{cygan,garofalo}) which yields the small-time asymptotics.
\end{proof}

\begin{remark}\label{densityRemark}
We note that $d_t(x_0)$ is independent of $x_0$ because of Corollary \ref{cor:SurjV}(b).
\end{remark}


In the sequel, we will use the notation $\mathcal{F}_I$ (defined with respect to the connection $D=\nabla^{\ve}$) and  $\Lambda_I(B)_t$, as introduced and discussed in appendix A.2.

\begin{corollary}\label{BCCorrolary}

It will hold that as $t \rightarrow 0$
\[
\mathbf{Str} (  p_{\mathcal{H},\ve}(t,x_0,x_0)) \sim d_t(x_0)  \mathbb{E}\left( \left.  \mathbf{Str} \left( \exp \left(\sum_{I,d(I) \leq n+2m} \Lambda_I(B)_t \mathcal{F}_I\right)(x_0) \right)\right\vert B_1= 0 \right) 
 \]
where $d_t(x_0)$ is the density at 0 of $V_t(x)$, as in Lemma \ref{densityLemma}.

\end{corollary}

\begin{proof}
Since $\mathcal H$ is two-step bracket generating, the homogeneous dimension is $Q = \dim \mathcal{H} + 2\dim\mathcal{V} = n + 2m$.  Taking $N = n + 2m$ in Theorem \ref{BCTheorem},  and applying similar arguments as in the proof of Proposition 4.2 in \cite{index},  the corollary follows by recognizing that for $|I| > 2, X_I$ is a linear combination of $X_i , [X_j,X_k]$ so that when $t \to 0$ the density at 0 of 
\[\sum_{I, d(I) \leq n+2m} \Lambda_I(B)_t X_I \]
is equivalent to $d_t(x_0)$ from the previous lemma.
\end{proof}

Applying the previous results we are now able to compute $\lim_{t \rightarrow 0} \mathbf{Str} (p_{\mathcal{H},\ve} (t,x_0,x_0))$. Choose local orthonormal basises $X_1, \dots, X_n$ and $Z_1, \dots, Z_m$ of respectively $\calH$ and $\calV$. 
\begin{lemma}
The integral
\begin{align*}
\mathcal{J} = \mathcal{J}(x_0)=\frac{2^m}{(2\pi )^{\frac{n}{2}+m} } \int_{\mathcal{V}_{x_0}} \det \left( \frac{\sqrt{J_z^* J_z}}{\sinh \sqrt{J_z^* J_z}}\right)^{1/2} dz,
\end{align*}
is a constant, so independent of the point $x_0 \in \M$ chosen. Furthermore, it holds that
\begin{align*} 
\lim_{t \to 0} \mathbf{Str}(p_{\mathcal{H},\ve} (t,x_0,x_0) ) =
\begin{cases}
\frac{\mathcal{J}}{\left( \frac{n}{2}+m\right)! }\mathbb{E} \left(\left. \mathbf{Str} \left[ A_{x_0}^{\frac{n}{2}+m} \right] \right\vert B_1=0\right), \quad \text{if n is even}\\ \\
0, \quad \text{if n is odd}.
\end{cases}
\end{align*}
where the random variable $A_{x_0}$ is given by
\begin{equation} \label{Adef1} A_{x_0} = - \frac{1}{2} \sum_{i,j,k,l=1}^n  \left(R_{kli}^j +\frac{1}{\ve} \sum_{r=1}^m T_{kl}^r T_{ij}^r \right)  a_i^* a_j^* a_l a_k+ \sum_{1 \le i<j \le n} \sum_{r,s=1}^m T_{ij;r}^s b_r^* b_s \int_0^1 B^i_t dB^j_t -B^j_t dB^i_t. \end{equation}
\end{lemma}
%
%
\begin{proof}
First, observe that \[\mathcal{J}(x_0) = (2t)^{\frac{n}{2}+m} d_t(x_0),\] and so the independence of $\mathcal{J}(x_0)$ from $x_0$ follows from Corollary \ref{cor:SurjV}(b) as in Remark \ref{densityRemark}.

Consider the expansion
\[\mathbf{Str} \left[\operatorname{exp}\left(\sum_{I,d(I) \leq n+2m} \Lambda_I(B)_t \mathcal{F}_I\right)(x_0)\right]\ = \sum_{k \geq 0}\frac{1}{k!}\mathbf{Str} \left[\left(\sum_{I,d(I) \leq n+2m} \Lambda_I(B)_t \mathcal{F}_I\right)^k(x_0)\right]. \]

 From the Weitzenb\"ock identity \eqref{WIdentity} we have for $i,j \in \{1,\dots,n+m\}$ that
\[\mathcal{F}_0 = -\scrR_\ve, \quad \mathcal{F}_i = 0, \quad \mathcal{F}_{(i,j)} = \hat{R}^\ve(Y_i,Y_j) \]
where $\{Y_1, \dots, Y_{n+m}\}$ form a local orthonormal frame and the $\{c_i,c^*_i\}_{i=1}^{n+m}$ form the associated Fermion calculus of $T\M$. Equation \eqref{scrR3} allows us to write
\[\scrR_\ve = \sum_{i,j ,k=1}^n  \langle \hat R^\ve (X_i, X_k) X_j, X_i \rangle_g a_k^* a_i + \sum_{i,j,k,l} \langle \hat R^\ve (X_i, X_j) X_k, X_l \rangle_g a_i^* a_j^* a_l a_k\]
where $\{a_i,a^*_i\}$ form the Fermion calculus for $\mathcal{H}$.

Recalling equation \eqref{supertraceformula} in the appendix we see that the supertrace will vanish for any term that is not of full degree; from our expressions for $\mathcal{F}_I$ it is thus clear that for $k < \frac{n}{2}+m$
\[\mathbf{Str} \left[\left( \sum_{I, d(I) \leq n+2m} \Lambda_I(B)_t \mathcal{F}_I \right)^k(x_0)\right] = 0.\]

Let us assume that $n$ is even. Applying the scaling property of Brownian motion, when $t \rightarrow 0$ the term $k = \frac{n}{2} + m$ will be dominant.  More precisely,
\begin{equation}\label{limitEQ1}
\begin{split}
\mathbb{E}&\left(  \left. \mathbf{Str} \left[\operatorname{exp}\left(\sum_{I,d(I) \leq n+2m} \Lambda_I(B)_t \mathcal{F}_I\right)(x_0)\right]  \right\vert B_1=0 \right)  \\
&= \frac{1}{\left(\frac{n}{2}+m\right)!}\mathbb{E}\left( \left. \mathbf{Str}\left[\left(\sum_{I, d(I) \leq n+2m} \Lambda_I(B)_t \mathcal{F}_I \right)^{\frac{n}{2}+m}(x_0)\right] \right\vert B_1=0 \right) + O\left(t^{\frac{n}{2}+m+\frac{1}{2}}\right). \\
\end{split}
\end{equation}
Then we have,
\begin{equation}\label{limitEQ2}
\begin{split}
\mathbb{E} & \left( \left. \mathbf{Str}\left[\left(\sum_{I, d(I) \leq n+2m} \Lambda_I(B)_t \mathcal{F}_I \right)^{\frac{n}{2}+m}(x_0)\right]  \right \vert B_1=0 \right) \\
& = \mathbb{E}\left( \left. \mathbf{Str}\left[\left( - t \scrR_\ve(x_0)+ \sum_{1 \le i<j \le n} \sum_{r,s=1}^s \hat{R}_{iir}^{\ve,s} b_r^* b_s \int_0^t B^i_u dB^j_u -B^j_u dB^i_u \right)^{\frac{n}{2}+m} \right] \right\vert B_1=0 \right) + O\left(t^{\frac{n}{2}+m+\frac{1}{2}}\right). \\
\end{split}
\end{equation}
We can then further simplify this expression using that by Lemma~\ref{lemma:hatRtoR}, Appendix, we know that $\hat R_{ijr}^{\ve,s} = R_{ijr}^s = T_{ij;r}^s$. We  also use \eqref{scrR3} and the fact that only the last term in $\mathscr{R}_\ve$ contributes to the supertrace. Combining Lemma \ref{densityLemma}, Corollary \ref{BCCorrolary}, and equations \eqref{limitEQ1} and \eqref{limitEQ2}, we apply the scaling property of Brownian motion again to find
\[\mathbf{Str} (p_{\mathcal{H},\ve} (t,x_0,x_0) ) =\frac{\mathcal{J}}{\left( \frac{n}{2}+m\right)! }\mathbb{E} \left(\mathbf{Str} \left[\left. A_{x_0}^{\frac{n}{2}+m} \right] \right\vert \\
B_1 = 0 \right) + O\left(t^{\frac{1}{2}}\right).\]
If $n$ is odd, we get by similar arguments that
\[\mathbf{Str} (p_{\mathcal{H},\ve} (t,x_0,x_0) ) = O\left(t^{\frac{1}{2}}\right).\]
completing the proof.
\end{proof}
%

In what follows, we will introduce the tensor $\scrT$ by
$$\scrT(Y_1, Y_2) = \hat R^\ve(\pi_{\calH} Y_1, Y_2) \pi_{\calV} = \pi_{\calV} \hat R^\ve(\pi_{\calH} Y_1, Y_2).$$
We observe that for any $X_1, X_2 \in \Gamma(\calH)$ and $Z \in \calV$, we have
$$\scrT(X_1, X_2) Z= (\nabla_Z T)(X_1, X_2) = \frac{1}{2\ve} \left( T(J_Z X_1, X_2) + T(X_1, J_Z X_2) \right),$$
where the latter equality follows from the symmetry condition of $\Delta_{\calH,\ve}$.
\begin{example}[H-type foliation] We again consider the case of the of H-type foliations as in Example~\ref{HtypeEx}. We recall that in this case, we have that $\Delta_{\calH,\ve}$ for $\ve = \frac{1}{\kappa}$. Let $x \in \M$ be a fixed point and let $\mathbf{Cl}(\calV_x)$ be the Clifford algebra of the vertical space. We remark that in this case, we have for any $u,v \in \calH_x$ with $v \in (\spn_{\zeta \in \mathbf{Cl}(\calV_x)} J_{\zeta} u)^\perp$, then we have $\scrT(u,v) = 0$. On the other hand if $v = J_\zeta u$, then for any $z \in \calV_x$,
$$\scrT(u, J_\zeta u)z = \kappa \pi_{\calV_x} (z \cdot \zeta^{\mathrm{odd}}),$$
where $\zeta^{\mathrm{odd}}$ is the odd part of $\zeta$ and $\pi_{\calV_x} \mathrm{Cl}(\calV_x) \to \calV_x$ is the projection to the first order part.

\end{example}

We can use the above definition and the previous lemma to prove the following.

\begin{proposition}\label{patodi}
Assume that $n$ or $m$ is odd, then
\[
\lim_{t \to 0}   \mathbf{Str}\text{ } (p_{\mathcal{H},\ve} (t,x,x) ) \, dx = 0
\]
Assume that both $n$ and $m$ are even, then
\[\lim_{t \to 0}   \mathbf{Str}\text{ } (p_{\mathcal{H},\ve} (t,x,x) ) \, dx =   \hat{\omega}_\mathcal{H}^\ve \wedge \left[  \det \left( \frac{\scrT}{\sinh (\scrT)}\right)^{1/2}\right]_m  \]

where $\left[ \cdot \right]_m $ denotes the $m$-form part and $\hat{\omega}_\mathcal{H}^\ve  $ is the horizontal Euler form, which is locally defined as 

\[
\hat{\omega}_\mathcal{H}^\ve=\frac{(-1)^{n/2} m!}{ 2^{n/2} \left(\frac{n}{2}+m \right)!} \mathcal{J} \sum_{\sigma,\tau
\in \mathfrak{S}_n} \epsilon (\sigma) \epsilon (\tau) \prod_{i=1}^{n-1}
\hat{R}^{\ve, \tau (i+1)}_{\sigma (i) \sigma (i+1) \tau (i) }dx_\mathcal{H}, 
\] 
where $\mathfrak{S}_n$ is the set of the permutations of the indices
$\{1,...,n\}$, $\epsilon$ the signature of a permutation, $\hat{R}^{\ve,l}_{ijk}$ is as in \eqref{indices} and $dx_\mathcal{H}$ the $n$-form $X_1^* \wedge \cdots \wedge X_n^*$.

\end{proposition}

\begin{proof}

We first assume that both $n$ and $m$ are even.  It remains to compute $\mathbb{E} \left(\mathbf{Str} \left[\left. A_{x_0}^{\frac{n}{2}+m} \right] \right\vert \\
B_1 = 0 \right)$.
Looking at \eqref{Adef1}, we have
\begin{align*}
& {\scriptstyle \mathbb{E} \left(\mathbf{Str} \left[\left. A_{x_0}^{\frac{n}{2}+m} \right] \right\vert  B_1 = 0 \right)} \\
=&{\scriptstyle \mathbf{Str} \left[ \left(-\sum_{i,j,k,l} \langle \hat R^\ve (X_i, X_j) X_k, X_l \rangle_g a_i^* a_j^* a_l a_k\right)^{n/2} \mathbb{E}\left[ \left. \left( \sum_{1 \le i<j \le n} \scrT (X_i, X_j)(x_0) \int_0^1 B^i_s dB^j_s -B^j_s dB^i_s \right)^{m} \right\vert B_1=0 \right] \right]}
\end{align*}

The term $\left(\sum_{i,j,k,l} \langle \hat R^\ve (X_i, X_j) X_k, X_l \rangle_g a_i^* a_j^* a_l a_k\right)^{n/2}$ is then analyzed as in the proof of Proposition 5.6 in \cite{ESAIM} (see also Lemma 2.35 in \cite{rosenberg}) and up to constant yields the horizontal Euler form $\hat{\omega}_\mathcal{H}^\ve$.  On the other hand, using again the formula for the supertrace, the term 
\[
\mathbb{E}\left[ \left. \left( \sum_{1 \le i<j \le n} \scrT (X_i, X_j)(x_0) \int_0^1 B^i_s dB^j_s -B^j_s dB^i_s \right)^{m} \right\vert B_1=0 \right]
\]
 can be replaced with 
\[
m! \, \mathbb{E} \left[ \left. \exp \left( \sum_{1 \le i<j \le n} \scrT (X_i, X_j)(x_0) \int_0^1 B^i_s dB^j_s -B^j_s dB^i_s \right) \right \vert B_1=0\right]
\]

and is analyzed using the L\'evy area formula as in the proof of Theorem 4.3 in \cite{index}: it yields the top degree  Fermionic piece of  $\det \left( \frac{\scrT}{\sinh (\scrT)}\right)^{1/2} (x_0) \in \mathbf{End} \left( \wedge
\mathcal{V}_{x_0}^\ast\right)$ (Fermionic calculus is done here on $\mathcal{V}_{x_0}$). 


If $n$ is even and $m$ is odd, a similar analysis shows that we  have  
\[
\mathbb{E} \left(\left. \mathbf{Str} \left[ A_{x_0}^{\frac{n}{2}+m} \right] \right\vert B_1=0\right)=0.
\]

\end{proof}

Combining Theorem \ref{PG result} and Proposition \ref{patodi} finally yields our main theorem:

\begin{theorem}
Assume that both $n$ and $m$ are even, then
\[
\chi (\M)= \int_\M    \hat{\omega}_\mathcal{H}^\ve \wedge \left[  \det \left( \frac{\scrT}{\sinh \scrT}\right)^{1/2}\right]_m . 
\]
Assume that $n$ or $m$ is odd, then $\chi (\M)=0$.
\end{theorem}

As a corollary, we obtain the following result:

\begin{corollary}
Assume that $\nabla J=0$, then $\chi (\M)=0$.
\end{corollary}

\begin{proof}
If $\nabla J=0$ then $\scrT=0$.
\end{proof}

\appendix

\section{Appendices}

\subsection{Fermion calculus and supertraces} \label{sec:Fermion}

In this section, we recall some basic elements of Fermion calculus, see section 2.2.2 in \cite{rosenberg} for more details. Let $V$ be a
$d$-dimensional Euclidean vector space. We denote $V^\ast$ its
dual and $\wedge V^\ast=\bigoplus_{k \ge 0} \wedge^k V^\ast$, its exterior algebra. If $u \in V^\ast$, we denote $a^\ast_u$ the
map $\wedge V^\ast \rightarrow \wedge V^\ast$, such that $a^\ast_u
(\omega)=u \wedge \omega$. The dual map is denoted $a_u$. Let now
$\theta_1$, ..., $\theta_d$ be an orthonormal basis of $V^\ast$.
We denote $a_i=a_{\theta_i}$.  If $I$ and $J$ are two words with
$1 \le i_1 < \cdots < i_k \le d$ and $1 \le j_1 < \cdots < j_l \le
d$, we denote
\[
A_{IJ}=a^\ast_{i_1} \cdots a^\ast_{i_k}   a_{j_1} \cdots a_{j_l}.
\]
The family of all the possible $A_{IJ}$ forms a basis of the
$2^{2d}$-dimensional vector space $\mathbf{End} \left( \wedge
V^\ast\right)$.  

If $A \in \mathbf{End} \left( \wedge
V^\ast\right)$, the supertrace $\mathbf{Str} (A)$ is the
difference of the trace of $A$ on even forms minus the trace of
$A$ on odd forms.  If $A = \sum_{I,J} c_{IJ} A_{IJ}$,
then we have
\begin{align}\label{supertraceformula}
\mathbf{Str} (A)=(-1)^{\frac{d(d-1)}{2}} c_{\{1,...,d\}\{1,...,d\}}.
\end{align}
In this paper, $c_{\{1,...,d\}\{1,...,d\}}$ will be called the top degree  Fermionic piece of $A$ and 
\[
[A]_d:= (-1)^{\frac{d(d-1)}{2}} c_{\{1,...,d\}\{1,...,d\}} \theta_1 \wedge \cdots \wedge \theta_d
\]
the $d$-form part of $A$.

\subsection{The Brownian Chen series parametrix method }

For the sake of completeness and to introduce some notations used in the paper, we reproduce here the essential ideas from \cite{bismutCR,ESAIM, index} to which we refer for further details.  Let $\mathcal{E}$ be a finite-dimensional vector bundle over a compact manifold  $\M$ equipped with a connection $D$ and consider a second order differential operator $\mathcal{L} = D_0 + \sum_{i=1}^d D_i^2$ with $D_i = \mathcal{F}_i + D_{X_i}$ for some smooth vector fields $X_i$ and potentials $\mathcal{F}_i$ on $\mathcal{E}$.  It is known that the differential equation 
\[\frac{\partial \Phi}{\partial t} = \mathcal{L}\Phi, \quad \Phi(0,x) = f(x)\]
has solution
\[\Phi(t,x) = (e^{t\mathcal{L}}f)(x) = P_tf(x).\]
At strongly regular points $x_0 \in \M$ it is furthermore true that $P_t$ admits a smooth heat kernel 
\[p_t(x_0, \cdot) \colon \mathbb{R}_{>0} \rightarrow \Gamma(\M,\operatorname{Hom}(\mathcal{E}))\]
\[t \mapsto p_t(x_0, \cdot)\]
which is to say
\[(P_tf)(x_0) := (e^{t\mathcal{L}}f)(x_0) = \int_\M p_t(x_0, y)f(y)\ dy.\]

We have a method of approximation for the heat kernel in this setting.
\begin{theorem}\label{BCTheorem}
Let $N \geq 1$ and define $(P_t^N f)(x) = \mathbb{E}(\Psi(1,x))$ where $\Psi(\tau, x)$ solves the random differential equation
\begin{equation}\label{PNt equation}
\frac{\partial \Psi}{\psi \tau} = \sum_{I \colon d(I) \leq N} \Lambda_I(B)_t (D_I \Psi)(\tau,x), \qquad \Psi(0,x) = f(x).
\end{equation}
where $I = (i_1,\dots,i_k) \in \{0,\dots,d\}^k$ is a word, $D_I = [D_{i_1},[\dots,[D_{i_{k-1}},D_{i_k}]\dots]]$, $d(I) = n(I)+k$ with $n(I)$ the number of 0's in $I$, and the random coefficients are defined by 
\[\Lambda_I(B)_t = 2^{d(I)/2} \sum_{\sigma \in \mathfrak{S}_k}\frac{(-1)^{e(\sigma)}}{k^2 \begin{pmatrix} k-1 \\ e(\sigma)\end{pmatrix}} \int_{\Delta^k[0,t]} \circ dB^{\sigma^{-1}(I)}\]
where $(B_t)_{t \ge 0}$ is a standard Brownian motion in $\mathbb R^d$.
Then 
\begin{enumerate}[$\bullet$]
\item For $k \geq 0$, define the norm
\[\|f\|_k = \sup_{0 \leq l \leq k} \sup_{0 \leq i_1, \dots, i_k} \sup_{x \in \M} \|D_{i_1}\cdots D_{i_l}f(x)\|.\]
It will hold that for any $k \geq 0$
\[\| P_tf - P_t^Nf\|_k = O\left( t^{\frac{N+1}{2}} \right), \qquad t \rightarrow 0\]
\item $P_t^N$ admits a smooth kernel $p_t^N$ such that for $N \geq 2$
\[p_t(x_0,x_0) = p_t^N(x_0,x_0) + O\left(t^{\frac{N+1-Q}{2}}\right), \qquad t \rightarrow 0\]
where $Q$ is the homogeneous dimension at $x_0$.
\item Write $\mathcal{F}_I = D_I - D_{X_I}$. For $N \geq 2$ it holds as $t \rightarrow 0$ that
\[p_t^N(x_0,x_0) = d_t^N(x_0)\mathbb{E}\left(\left.\operatorname{exp}\left(\sum_{I,d(I) \leq N}\Lambda_I(B)_t\mathcal{F}_I \right) (x_0) \right\vert \sum_{I, d(I) \leq N}\Lambda_I(B)_t X_I(x_0) = 0 \right)+ O\left(t^{\frac{N+1-Q}{2}}\right)\]
where $d_t^N(x)$ is the density at 0 of the random variable $\sum_{I, d(I) \leq N}\Lambda_I(B)_tX_I(x)$.
\end{enumerate}
\end{theorem}

We refer to \cite{bismutCR} and \cite[Section 5.1]{ESAIM} for the proofs and further details, but we remark that roughly the theorem says that in small time we can approximate the heat kernel of $\mathcal{L}$ by the kernel associated to solutions of equation \ref{PNt equation}, for which we will be able to say much more.

\subsection{Curvature of the connection $\hat\nabla^\ve$}
We want to give details on writing the curvatures of $\hat \nabla^\ve$ in terms of the Bott connection $\nabla$.
\begin{lemma} \label{lemma:hatRtoR}
Relative to the notation of \eqref{indices} we have the following identities. Recall that $i,j,k,l$ denotes vector fields from a basis of $\calH$, while indices $r,s$ denotes such elements from a basis of $\calV$
\begin{enumerate}[\rm (i)]
\item $R_{ijk}^l = R_{kli}^j$, $R_{r_1 s_1 r_1}^{s_2} = R_{r_2 s_2 r_1}^{s_1}$,
\item $R_{ijr}^s = T_{ij;r}^s$, $R_{irk}^l =0$, $R_{is_1 r_2}^{s_2} =0$,
\item $T_{ij;r}^r =0$. Equivalently $(\nabla_Z J)_Z =0$ for any vector field $Z$ with values in $\calV$.
\item $\hat R_{ijk}^{\ve,l} = R_{ijk}^l + \frac{1}{\ve} \sum_{s=1}^m T_{ij}^s T_{kl}^s$.
\item $\hat R_{irk}^{\ve,l} = \frac{1}{\ve} T_{kl;i}^s$.
\item $\hat R_{rsk}^{\ve,l}= \frac{2}{\ve} T_{kl;r}^s + \frac{1}{\ve^2} \sum_{i=1}^n (T_{il}^r T_{ki}^s - T_{il}^s T_{ki}^r) $
\end{enumerate}
\end{lemma}
\begin{proof}
From \eqref{hatnablave}, we observe that
\begin{equation} \label{hatRtoR} \hat R^\ve(X,Y) Z = R(X,Y) Z+ \frac{1}{\ve} (\nabla_X J)_Y Z - \frac{1}{\ve} (\nabla_Y J)_X Z + \frac{1}{\ve} J_{T(X,Y)}Z +\frac{1}{\ve^2} [J_X, J_Y]Z.\end{equation}
We will also use the first Bianchi identity for connections with torsion
$$\circlearrowright R(X,Y) Z = \circlearrowright (\nabla_X T)(X, Y) + \circlearrowright T(T(X, Y),Z),$$
where $\circlearrowright$ denotes the cyclic sum. We furthermore observe the following identities.
\begin{enumerate}[\rm (i)]
\item Since $\langle T(Y_1,Y_2), Y_3 \rangle$ and $T(T(Y_1, Y_2), Y_3)$ vanishes if $Y_1, Y_2, Y_3$ are either all vertical or all horizontal, 
\begin{align*}
\langle R(X_1, X_2) X_3, X_4 \rangle_g & = \langle R(X_3, X_4) X_1, X_2 \rangle_g, \\
\langle R(Z_1, Z_2) Z_3, Z_4 \rangle_g & = \langle R(Z_3, Z_4) Z_1, Z_2 \rangle_g,
\end{align*}
for any $X_i \in \Gamma(\calH)$, $Z_i \in \Gamma(\calV)$, $i=1,2,3, 4$.
\item From \cite[Appendix A]{Gro20}, we know that for $X_1, X_2 \in \Gamma(\calH)$, $Z_1, Z_2 \in \Gamma(\calV)$,
$$R(X_1, X_2) Z_1 = (\nabla_{Z_1} T)(X_1, X_2),\quad R(X_1, Z_1) X_2 =0 \quad R(X_1, Z_1) Z_2 =0.$$
\item Since $\nabla$ is compatible with the metric then $(\nabla_Z J)_Z =0$ for any $Z \in \Gamma(\calV)$, as for any $X_1, X_2 \in \Gamma(\calH)$,
\begin{align*}
0 = \langle Z, R(X_1, X_2) Z \rangle_g& =  \langle Z, \circlearrowright R(X_1, X_2) Z \rangle_g \\
& = \langle Z, (\nabla_Z T)(X_1, X_2) \rangle_g = \langle X_2, (\nabla_Z J)_Z X_1 \rangle_g.
\end{align*}
\item We observe first that from \eqref{hatRtoR}, for any $X_1, X_2, X_3, X_4 \in \Gamma(\calH)$
\begin{align*}
& \langle \hat R^\ve(X_1, X_2) X_3, X_4 \rangle_g = \langle R(X_1, X_2) X_3, X_4 \rangle_g +\frac{1}{\ve} \langle J_{T(X_1, X_2)} X_3, X_4 \rangle_g \\
& \stackrel{\rm (i)}{=} \langle R(X_3, X_4) X_1, X_2 \rangle_g +\frac{1}{\ve} \langle T(X_1, X_2) , T(X_3, X_4) \rangle_g.\end{align*}
\item Next, for any $X_1, X_2 \in \Gamma(\calH)$, $Z \in \Gamma(\calV)$,
$$\hat R^\ve(X_1, Z) X_2 \stackrel{\rm (ii)}{=}  \frac{1}{\ve} (\nabla_{X_1} J)_{Z} X_2.$$
\item For the final property observe that
\begin{align*}
R(Z_1, Z_2) X_1 & \stackrel{\rm (ii)}{=} \circlearrowright R(Z_1, Z_2) X_1 = 0.
\end{align*}
Hence
\begin{align*}
\hat R^\ve(Z_1, Z_2) X_1 & = \frac{1}{\ve} (\nabla_{Z_1} J)_{Z_2} X_1 - \frac{1}{\ve} (\nabla_{Z_2} J)_{Z_1} X_1 + \frac{1}{\ve^2} [J_{Z_1}, J_{Z_2} ] X_1 \\
& \stackrel{\rm (iii)}{=} \frac{2}{\ve} (\nabla_{Z_1} J)_{Z_2} X_1 + \frac{1}{\ve^2} [J_{Z_1}, J_{Z_2} ] X_1. \qedhere
\end{align*}
\end{enumerate}
\end{proof}

\bibliographystyle{abbrv}
\bibliography{Bibliography}

\end{document}